\documentclass[12pt,reqno]{article}
\usepackage[usenames]{color}
\usepackage[colorlinks=true,linkcolor=webgreen,filecolor=webbrown,citecolor=webgreen]{hyperref}
\usepackage{amssymb}
\usepackage{amsmath}
\usepackage{amsfonts}
\usepackage[T1]{fontenc}
\usepackage{enumerate}
\usepackage{graphicx}
\usepackage{color}

\definecolor{webgreen}{rgb}{0,.5,0}
\definecolor{webbrown}{rgb}{.6,0,0}
\newtheorem{theorem}{Theorem}

\newtheorem{corollary}[theorem]{Corollary}

\newtheorem{remark}[theorem]{Remark}

\newenvironment{proof}[1][Proof]{\noindent\textbf{#1.} }{\ \rule{0.5em}{0.5em}}

\allowdisplaybreaks

\begin{document}

\begin{center}
\vskip1cm

{\LARGE \textbf{Kernels for composition of positive linear operators}}

\vspace{2cm}

{\large Ulrich Abel}\\[0pt]
\textit{Technische Hochschule Mittelhessen}\\[0pt]
\textit{Fachbereich MND}\\[0pt]
\textit{Wilhelm-Leuschner-Stra\ss e 13, 61169 Friedberg }\\[0pt]
\textit{Germany}\\[0pt]
\href{mailto:Ulrich.Abel@mnd.thm.de}{\texttt{Ulrich.Abel@mnd.thm.de}}

\vspace{1cm}

{\large Ana Maria Acu}\\[0pt]
\textit{Lucian Blaga University of Sibiu}\\[0pt]
\textit{Department of Mathematics and Informatics}\\[0pt]
\textit{Romania}\\[0pt]
\textit{Str. Dr. Ion Ra\c tiu, No. 5-7, 550012 Sibiu}\\[0pt]
\textit{Romania}\\[0pt]
\href{mailto:anamaria.acu@ulbsibiu.ro}{\texttt{anamaria.acu@ulbsibiu.ro}}

\vspace{1cm}

{\large Margareta Heilmann}\\[0pt]
\textit{University of Wuppertal}\\[0pt]
\textit{School of Mathematics and Natural Sciences}\\[0pt]
\textit{Gau\ss straße 20, 42119 Wuppertal}\\[0pt]
\textit{Germany}\\[0pt]
\href{mailto:heilmann@math.uni-wuppertal.de}{\texttt{heilmann@math.uni-wuppertal.de}}

\vspace{1cm}

{\large Ioan Ra\c sa}\\[0pt]
\textit{Technical University of Cluj-Napoca}\\[0pt]
\textit{Faculty of Automation and Computer Science, Department of Mathematics}\\[0pt]
\textit{Str. Memorandumului nr. 28, 400114 Cluj-Napoca}\\[0pt]
\textit{Romania}\\[0pt]
\href{mailto:ioan.rasa@math.utcluj.ro}{\texttt{ioan.rasa@math.utcluj.ro}}

\end{center}

\vspace{2cm}

{\large \textbf{Abstract.}}

\bigskip

This paper investigates the composition of Bernstein--Durrmeyer operators and Szász--Mirakjan--Durrmeyer operators, focusing on the structure and properties of the associated kernel functions. In the case of the Bernstein--Durrmeyer operators, we establish new identities for the kernel arising from the composition of two and three operators, from which the commutativity of these operators follows naturally. Building on the eigenstructure of the Bernstein–Durrmeyer operator $M_n$, we obtain a representation of the iterate $M_n^r$ as a linear combination of the operators $M_k$, for $k=0,1,\dots,n$. We also address the composition of Szász--Mirakjan--Durrmeyer operators and revisit a known result giving an elementary  proof. 

\bigskip

\smallskip \emph{Mathematics Subject Classification (2020):} 41A36 %% Approximation by positive operators.

\smallskip \emph{Keywords:} Approximation by positive operators, 
composition of positive linear operators, iterates, kernels of integral operators, Bernstein--Durrmeyer operators, Szász--Mirakjan--Durrmeyer operators
\vspace{2cm}

%%%%%%%%%%%%%%%%%%%%%%%%%%%%%%%%%%%%%%%%%%%%%%%%%%%%%%%%%%%%%%%%%%%

\section{Introduction}

\label{intro}

Integral operators play a central role in approximation theory, functional
analysis, and various applications in numerical analysis. Classical examples
of such operators include the Bernstein-Durrmeyer and Sz\'{a}%
sz-Mirakjan-Durrmeyer operators, whose properties have been extensively
studied in the literature.
% (see, for example, \cite{1, 7, 3,8,9,6,2} and the
%references therein). In recent years, these operators have also found
%applications in solving real-world problems, particularly in connection with
%sampling theory (see, e.g., \cite{5}, \cite{4}). 
Within the framework of
integral operators, key topics of interest include their kernels, iterates,
and compositions.

In this paper, we investigate a family of integral operators defined by 
\begin{equation*}
(L_n f)(x) = \int_I f(y) K_n(x,y) dy
\end{equation*}
with $I=[0,1]$ or $I=[0,\infty)$ and $K_n(x,y)$ a given kernel function.

A key object of interest in our study is the composition of $r$ such
operators. Specifically, we consider the composition 
\begin{equation*}
((L_{n_{r}}\circ \cdots \circ L_{n_{1}})f)(x)=\int_{I}f(t_{1})K_{n_{r},\dots
,n_{1}}(t_{r+1},t_{1})dt_{1}
\end{equation*}%
where the kernel is recursively given by 
\begin{equation*}
K_{n_{r},\dots ,n_{1}}(t_{r+1},t_{1})=\int_{I}K_{n_{r-1},\dots
,n_{1}}(t_{r},t_{1})\cdot K_{n_{r}}(t_{r+1},t_{r})dt_{r}.
\end{equation*}

In Section \ref{sec2} we consider the Bernstein--Durrmeyer operators $M_n$
and derive new nice identities for the kernel function in the case of
composing two respectively three of these operators. From these identities
one deduces immediately that involved operators commute. Section \ref{sec3}
is devoted to a new approach, using the eigenstructure of $M_n$, for
studying the kernel function when $r$ Bernstein--Durrmeyer operators are
composed.  As a by-product we get a representation of the iterate $M_n^r$ as
a linear combination of $M_k$, $k=0,1,\dots,n$. In Section \ref{sec4} we
deal with the Durrmeyer modification of the Sz\'{a}sz--Mirakjan operators
and present an elementary proof of a result given by Abel and Ivan in \cite%
{Abel-Ivan-Varna-CTF2005-proc-2006}. From this result it follows in
particular that composing two Sz\'{a}sz--Mirakjan--Durrmeyer operators we
get an operator from the same family.

\section{The case of Bernstein--Durrmeyer operators}

\label{sec2}

Let $L_{1}[0,1]$ denote the space of Lebesgue integrable functions on the
interval $[0,1]$ and $\Pi _{n}$ the space of polynomials of degree at most $%
n $.

The so called Bernstein--Durrmeyer operators $M_{n}:L_{1}[0,1]%
\longrightarrow \Pi _{n}$ are defined by 
\begin{eqnarray*}
(M_{n}f)(x) &=&\int_{0}^{1}f(y)K_{n}(x,y)dy,\, \,x\in \lbrack 0,1], \\
K_{n}(x,y) &=&(n+1)\sum_{k=0}^{n}p_{n,k}(x)p_{n,k}(y),\,p_{n,k}(x)=\binom{n}{%
k}x^{k}(1-x)^{n-k}.
\end{eqnarray*}%
In 1981 Derriennic \cite[Th\'{e}or\`{e}me III.3]{De1981} proved that the
eigenfunctions of the Bernstein-Durrmeyer operators are the
(orthonormalized) Legendre polynomials on the interval $[0,1]$, i.e., 
\begin{equation*}
Q_{0}(x)=1,\,Q_{k}(x)=\frac{\sqrt{2k+1}}{k!}\left( \frac{d}{dx}\right)
^{k}\left( x^{k}\left( 1-x\right) ^{k}\right)
\end{equation*}%
with corresponding eigenvalues 
\begin{equation*}
\lambda _{n,k}=\frac{n^{\underline{k}}}{\left( n+k+1\right) ^{\underline{k}}}
\end{equation*}%
and deduced the representation of the operators in terms of these
eigenfunctions, i.e., 
\begin{equation*}
M_{n}f=\sum_{k=0}^{n}\lambda _{n,k}\left \langle Q_{k},f\right \rangle Q_{k}
\end{equation*}%
with the inner product $\left \langle f,g\right \rangle
=\int_{0}^{1}f(t)g(t)dt $.

Ditzian and Ivanov \cite{DiIv1989} remarked that from this result it follows
immediately that the operators commute. In fact, we have 
\begin{equation*}
(M_{m}\circ M_{n})f=\sum_{k=0}^{\min \{m,n\}}\lambda _{m,k}\lambda
_{n,k}\left \langle Q_{k},f\right \rangle Q_{k}.
\end{equation*}

It is possible to represent $K_{m,n}\left( x,y\right)$ in terms of Bernstein basis polynomials. An application of Leibniz' rule for derivatives yields 
\begin{equation}  \label{eq.M1}
Q_{k}\left( x\right) =\sqrt{2k+1}\sum_{j=0}^{k}\left( -1\right) ^{j}\binom{k 
}{j}p_{k,j}\left( x\right).
\end{equation}
Starting from 
\begin{equation*}
K_{m,n}\left( x,y\right) =\sum_{k=0}^{\min \{m,n\}}\lambda _{m,k}\lambda
_{n,k}Q_{k}\left( x\right) Q_{k}\left( y\right)
\end{equation*}%
and using (\ref{eq.M1}), we obtain 
\begin{eqnarray*}
K_{m,n}\left( x,y\right) &=&\sum_{k=0}^{\min \left \{ m,n\right \} }\frac{m^{%
\underline{k}}}{\left( m+k+1\right) ^{\underline{k}}}\frac{n^{\underline{k}}%
}{\left( n+k+1\right) ^{\underline{k}}}\left( 2k+1\right) \\
&&\times \sum_{i=0}^{k}\left( -1\right) ^{i}\binom{k}{i}p_{k,i}\left(
x\right) \sum_{j=0}^{k}\left( -1\right) ^{j}\binom{k}{j}p_{k,j}\left(
y\right) .
\end{eqnarray*}
In this representation all possible products $p_{k,i}\left(x\right) p_{k,j}\left(y \right), \ 0\le i,j \le k \le \min \{m,n\}$, appear.

Now we are going to derive a different representation of $K_{m,n}\left(x,y\right) $, which contains only products 
$p_{k,\ell}\left(x\right) p_{k,\ell}\left(y \right), \ 0\le \ell \le k \le \min \{m,n\}$. 
To our best knowledge it is new. 

\begin{theorem}
\label{P1} \label{corollary-kernel-K_m,n[-1]}For $m,n\geq 0$, the kernel of
the composition of classical Bernstein--Durrmeyer operators $M_{m}\circ
M_{n} $ has the representation 
\begin{equation*}
K_{m,n}\left( x,y\right) =\frac{\left( m+1\right) !\left( n+1\right) !}{%
\left( m+n+1\right) !}\sum_{k=0}^{\min \{m,n\}}\binom{m}{k}\binom{n}{k}%
\sum_{\ell =0}^{k}p_{k,\ell }\left( x\right) p_{k,\ell }\left( y\right) .
\end{equation*}
\end{theorem}

In our calculations we will use the following well-known formulas 
\begin{align}
& \displaystyle \sum_{k=0}^{n}p_{n,k}(x)=1,  \label{eq.w1} \\
& \int_{0}^{1}p_{n,k}\left( t\right) dt=\frac{1}{n+1}, \text{ \quad } k=0,\ldots,n ,  \label{eq.w2} \\
& p_{m,\mu }\left( t\right) p_{n,\nu }\left( t\right) =\frac{\binom{m}{\mu }%
\binom{n}{\nu }}{\binom{m+n}{\mu +\nu }}p_{m+n,\mu +\nu }\left( t\right) .
\label{eq.w3}
\end{align}
Furthermore, we make frequent use of the falling factorials denoted by $z^{\underline{0}}=1$, $z^{\underline{r}}=z\left( z-1\right) \cdots
\left( z-r+1\right) $, for $r\in \mathbb{N}$.

\begin{proof}[{Proof of Theorem \protect\ref{corollary-kernel-K_m,n[-1]}}]
Using (\ref{eq.w2}) and (\ref{eq.w3}), we obtain 
\begin{eqnarray*}
K_{m,n}\left( x,y\right) &=&\left( m+1\right) \left( n+1\right) \sum_{\mu
=0}^{m}p_{m,\mu }\left( x\right) \sum_{\nu =0}^{n}p_{n,\nu }\left( y\right)
\int_{0}^{1}p_{m,\mu }\left( t\right) p_{n,\nu }\left( t\right) dt \\
&=&\frac{\left( m+1\right) \left( n+1\right) }{m+n+1}\sum_{\mu
=0}^{m}p_{m,\mu }\left( x\right) \sum_{\nu =0}^{n}p_{n,\nu }\left( y\right) 
\frac{\binom{m}{\mu }\binom{n}{\nu }}{\binom{m+n}{\mu +\nu }} \\
&=&\frac{\left( m+1\right) !\left( n+1\right) !}{\left( m+n+1\right) !}%
\sum_{\mu =0}^{m}p_{m,\mu }\left( x\right) \dfrac{1}{(m-\mu )!\mu !} \\
&&\times \sum_{\nu =0}^{n}p_{n,\nu }\left( y\right) (m+n-\mu -\nu )^{%
\underline{m-\mu }}(\mu +\nu )^{\underline{\mu }}.
\end{eqnarray*}%
For the inner sum of $K_{m,n}\left( x,y\right) $ we have 
\begin{eqnarray}
&&\sum_{\nu =0}^{n}p_{n,\nu }\left( y\right) (m+n-\mu -\nu )^{\underline{%
m-\mu }}(\mu +\nu )^{\underline{\mu }}  \label{eq.ww1} \\
&=&\left( \frac{\partial }{\partial s}\right) ^{m-\mu }\left( \frac{\partial 
}{\partial t}\right) ^{\mu }\left. s^{m-\mu }t^{\mu }\sum_{\nu
=0}^{n}p_{n,\nu }\left( y\right) s^{n-\nu }t^{\nu }\right \vert _{s=t=1} 
\notag \\
&=&\left. \left( \frac{\partial }{\partial s}\right) ^{m-\mu }\left( \frac{%
\partial }{\partial t}\right) ^{\mu }\left( s^{m-\mu }t^{\mu }\left[
ty+s\left( 1-y\right) \right] ^{n}\right) \right \vert _{s=t=1}  \notag \\
&=&\sum_{i=0}^{\mu }\binom{\mu }{i}\sum_{j=0}^{m-\mu }\binom{m-\mu }{j}\frac{%
\mu !}{i!}\frac{\left( m-\mu \right) !}{j!}
n^{\underline{i+j}}   %% Change 25.06.2025
y^{i}\left( 1-y\right) ^{j}  \notag \\
&=&\displaystyle \sum_{i=0}^{\mu }{\binom{\mu }{i}}\sum_{j=0}^{m-\mu }{%
\binom{m-\mu }{j}}{\binom{n}{i+j}}\mu !(m-\mu )!p_{i+j,i}(y).  \notag
\end{eqnarray}%
Thus we get for the kernel 
\begin{eqnarray*}
&&K_{m,n}(x,y) \\
&=&\frac{(m+1)!(n+1)!}{(m+n+1)!}\sum_{\mu =0}^{m}p_{m,\mu
}(x)\sum_{i=0}^{\mu }{\binom{\mu }{i}}\sum_{j=0}^{m-\mu }{\binom{m-\mu }{j}}{%
\binom{n}{i+j}}p_{i+j,i}(y) \\
&=&\frac{(m+1)!(n+1)!}{(m+n+1)!}\sum_{i=0}^{m}\sum_{j=0}^{m-i}{\binom{n}{i+j}%
}p_{i+j,i}(y)\sum_{\mu =i}^{m-j}{\binom{\mu }{i}}{\binom{m-\mu }{j}}p_{m,\mu
}(x).
\end{eqnarray*}%
For the inner sum by using (\ref{eq.w1}) one has 
\begin{align*}
\sum_{\mu =i}^{m-j}{\binom{\mu }{i}}{\binom{m-\mu }{j}}p_{m,\mu }(x)& =%
\dfrac{m^{\underline{i+j}}}{i!j!} x^{i}(1-x)^{j} \sum_{\mu =i}^{m-j}p_{m-j-i,\mu -i}(x)
%% Change 25.06.2025
\\
& ={\binom{m}{i+j}}p_{i+j,i}(x).
\end{align*}%
Therefore, 
\begin{equation*}
K_{m,n}(x,y)=\frac{(m+1)!(n+1)!}{(m+n+1)!}\sum_{i=0}^{m}\sum_{j=0}^{m-i}{%
\binom{n}{i+j}}p_{i+j,i}(y){\binom{m}{i+j}}p_{i+j,i}(x).
\end{equation*}%
Collecting all terms with $i+j=k$ we get our proposition.
\end{proof}

\begin{remark}
Starting from the definition of the operators we have 
\begin{eqnarray*}
\lefteqn{K_{n_{r},\dots ,n_{1}}(t_{r+1},t_{1})} \\
&=&\int_{0}^{1}\cdots \int_{0}^{1}\sum_{k_{1}\geq 0}\cdots \sum_{k_{r}\geq
0}p_{n_{r},k_{r}}\left( t_{r+1}\right) p_{n_{r},k_{r}}\left( t_{r}\right)
\cdots p_{n_{1},k_{1}}\left( t_{2}\right) p_{n_{1},k_{1}}\left( t_{1}\right)
dt_{2}\cdots dt_{r}.
\end{eqnarray*}
Using elementary calculations, we obtain the following representation of the
kernel 
\begin{equation*}
K_{n_{r},\dots ,n_{1}}(x,y)=(n_{r}+1)\left( \prod_{i=1}^{r-1}\dfrac{n_{i}+1}{
n_{i}+n_{i+1}+1}\right) \sum_{k_{1}\geq 0}\cdots \sum_{k_{r}\geq 0}{\binom{
n_{r}}{k_{r}}}^{2}
\end{equation*}
\begin{equation*}
\times \left( \prod_{i=1}^{r-1}{\binom{n_{i}}{k_{i}}}^{2}{\binom{
n_{i}+n_{i+1}}{k_{i}+k_{i+1}}}^{-1}\right)
y^{k_{1}}(1-y)^{n_{1}-k_{1}}x^{k_{r}}(1-x)^{n_{r}-k_{r}}.
\end{equation*}
\end{remark}

Next we consider the composition of three operators and derive the following
nice representation for the kernel.

\begin{theorem}
For $r=3$ the kernel of the composition of classical Bernstein--Durrmeyer $%
M_{n_3}\circ M_{n_2}\circ M_{n_1}$ operators has the representation 
\begin{align*}
K_{n_3,n_2,n_1} \left( x,y\right) &=\frac{\left( n_3+1\right) !\left(
n_2+1\right) !\left( n_1+1\right) !\left( n_3+n_2+n_1+1\right) !}{\left(
n_3+n_2+1\right) !\left( n_3+n_1+1\right) !\left( n_2+n_1+1\right) !} \\
&\times \sum_{k= 0}^{\min \{n_1,n_2,n_3\}}\frac{\binom{n_3}{k}\binom{n_2}{k}%
\binom{n_1}{k}}{\binom{n_3+n_2+n_1+1}{k}}\sum_{\ell =0}^{k}p_{k,\ell }\left(
x\right) p_{k,\ell }\left( y\right).
\end{align*}
\end{theorem}

\begin{proof}
By Theorem \ref{P1}, we have 
\begin{equation*}
K_{n_{2},n_{1}}(t,y)=\dfrac{(n_{2}+1)!(n_{1}+1)!}{(n_{2}+n_{1}+1)!}%
\displaystyle \sum_{k\geq 0}{\binom{n_{2}}{k}}{\binom{n_{1}}{k}}\sum_{\ell
=0}^{k}p_{k,\ell }(t)p_{k,\ell }(y).
\end{equation*}%
Using the recurrence formula for the kernel and Theorem \ref{P1} we get 
\begin{align*}
K_{n_{3},n_{2},n_{1}}(x,y)& =\displaystyle%
\int_{0}^{1}K_{n_{2},n_{1}}(t,y)K_{n_{3}}(x,t)dt \\
& =\displaystyle \int_{0}^{1}\dfrac{(n_{2}+1)!(n_{1}+1)!}{(n_{2}+n_{1}+1)!}%
\sum_{k\geq 0}{\binom{n_{2}}{k}}{\binom{n_{1}}{k}}\sum_{\ell
=0}^{k}p_{k,\ell }(t)p_{k,\ell }(y) \\
& \times (n_{3}+1)\sum_{\mu =0}^{n_{3}}p_{n_{3},\mu }(x)p_{n_{3},\mu }(t)dt
\\
& =\dfrac{(n_{2}+1)!(n_{1}+1)!}{(n_{2}+n_{1}+1)!}\sum_{k\geq 0}{\binom{n_{2}%
}{k}}{\binom{n_{1}}{k}}\sum_{\ell =0}^{k}p_{k,\ell }(y) \\
& \times (n_{3}+1)\sum_{\mu =0}^{\mu }p_{n_{3},\mu }(x)\int_{0}^{1}p_{k,\ell
}(t)p_{n_{3},\mu }(t)dt.
\end{align*}%
Using (\ref{eq.w2}) and (\ref{eq.w3}) we obtain 
\begin{eqnarray*}
\lefteqn{K_{n_{3},n_{2},n_{1}}(x,y)} \\
&=&\dfrac{(n_{3}+1)!(n_{2}+1)!(n_{1}+1)!}{(n_{2}+n_{1}+1)!}\sum_{k\geq 0}{%
\binom{n_{2}}{k}}{\binom{n_{1}}{k}}\sum_{\ell =0}^{k}p_{k,\ell }(y){\binom{k%
}{\ell }}\dfrac{1}{(k+n_{3}+1)!}\cdot S_{1},
\end{eqnarray*}%
where 
\begin{equation*}
S_{1}:=\displaystyle \sum_{\mu =0}^{n_{3}}p_{n_{3},\mu }(x)(k+n_{3}-\ell
-\mu )^{\underline{k-l}}(\ell +\mu )^{\underline{\ell }}.
\end{equation*}

In the same way as in (\ref{eq.ww1}) we derive 
\begin{equation*}
S_{1}=\displaystyle \sum_{i=0}^{\ell }{\binom{\ell }{i}}\sum_{j=0}^{k-\ell }{%
\binom{k-\ell }{j}}{\binom{n_{3}}{i+j}}\ell !(k-\ell )!p_{i+j,i}\left(
x\right) .
\end{equation*}%
Thus, by interchanging the order of summation we get 
\begin{align*}
K_{n_{3},n_{2},n_{1}}(x,y)& =\dfrac{(n_{3}+1)!(n_{2}+1)!(n_{1}+1)!}{%
(n_{2}+n_{1}+1)!} \\
& \times \sum_{k\geq 0}{\binom{n_{2}}{k}}{\binom{n_{1}}{k}}\dfrac{k!}{%
(n_{3}+k+1)!}\sum_{i=0}^{k}\dfrac{1}{i!}\sum_{j=0}^{k-i}\dfrac{1}{j!}{\binom{%
n_{3}}{i+j}}p_{i+j,i}(x)S_{2},
\end{align*}%
where 
\begin{align*}
S_{2}& :=\displaystyle \sum_{\ell =i}^{k-j}p_{k,\ell }(y)\ell ^{\underline{i}%
}(k-\ell )^{\underline{j}}=\dfrac{k!}{(k-j-i)!}y^{i}(1-y)^{j} \\
& ={\binom{k}{i+j}}i!j!p_{i+j,i}(y).
\end{align*}%
Now 
\begin{align*}
K_{n_{3},n_{2},n_{1}}(x,y)& =\dfrac{(n_{3}+1)!(n_{2}+1)!(n_{1}+1)!}{%
(n_{2}+n_{1}+1)!}\sum_{k\geq 0}{\binom{n_{2}}{k}}{\binom{n_{1}}{k}}\dfrac{k!%
}{(n_{3}+k+1)!} \\
& \times \sum_{i=0}^{k}\sum_{j=0}^{k-i}{\binom{n_{3}}{i+j}}p_{i+j,i}(x){%
\binom{k}{i+j}}p_{i+j,i}(y).
\end{align*}%
Collecting all terms with $i+j=\ell $ and interchanging the order of
summation leads to 
\begin{align*}
K_{n_{3},n_{2},n_{1}}(x,y)& =\dfrac{(n_{3}+1)!(n_{2}+1)!(n_{1}+1)!}{%
(n_{2}+n_{1}+1)!}\sum_{k\geq 0}{\binom{n_{2}}{k}}{\binom{n_{1}}{k}}\dfrac{k!%
}{(n_{3}+k+1)!} \\
& \times \sum_{\ell =0}^{k}{\binom{n_{3}}{\ell }}{\binom{k}{\ell }}\sum_{\nu
=0}^{\ell }p_{\ell ,\nu }(x)p_{\ell ,\nu }(y) \\
& =\dfrac{(n_{3}+1)!(n_{2}+1)!(n_{1}+1)!}{(n_{2}+n_{1}+1)!}\sum_{\ell \geq 0}%
{\binom{n_{3}}{\ell }}\sum_{\nu =0}^{\ell }p_{\ell ,\nu }(x)p_{\ell ,\nu }(y)
\\
& \times \sum_{k\geq \ell }{\binom{n_{2}}{k}}{\binom{n_{1}}{k}}\dfrac{k!}{%
(n_{3}+k+1)!}{\binom{k}{\ell }}.
\end{align*}%
For the inner sum we get

\begin{eqnarray*}
&&\sum_{k\geq \ell }{\binom{n_{2}}{k}}{\binom{n_{1}}{k}}\dfrac{k!}{%
(k+n_{3}+1)!}{\binom{k}{\ell }} \\
&=&\dfrac{n_{2}!n_{1}!}{\ell !}\dfrac{1}{(n_{2}+n_{3}+1)!(n_{1}-\ell )!}%
\sum_{k}{\binom{n_{1}-\ell }{k}}{\binom{n_{2}+n_{3}+1}{n_{2}-k-\ell }} \\
&=&\dfrac{n_{2}!n_{1}!}{\ell !(n_{1}-\ell )!(n_{2}+n_{3}+1)!}{\binom{%
n_{1}+n_{2}+n_{3}-\ell +1}{n_{2}-\ell }}.
\end{eqnarray*}%
Therefore, 
\begin{align*}
K_{n_{3},n_{2},n_{1}}(x,y)& =\dfrac{(n_{3}+1)!(n_{2}+1)!(n_{1}+1)!}{%
(n_{3}+n_{2}+1)!(n_{3}+n_{1}+1)!(n_{2}+n_{1}+1)!} \\
& \times \sum_{\ell \geq 0}(n_{1}+n_{2}+n_{3}-\ell +1)!\ell !{\binom{n_{3}}{%
\ell }}{\binom{n_{2}}{\ell }}{\binom{n_{1}}{\ell }}\sum_{\nu =0}^{\ell
}p_{\ell ,\nu }(x)p_{\ell ,\nu }(y)
\end{align*}%
and the proof is complete.
\end{proof}

\section{Composition of several Bernstein--Durrmeyer operators}

\label{sec3}

The purpose of this section is to show that a representation of the form 
\begin{equation*}
K_{n_{r},\ldots ,n_{1}}\left( x,y\right) =\displaystyle \sum_{k=0}^{n}a_{k}%
\left( n_{r},\ldots ,n_{1}\right) \sum_{\ell =0}^{k}p_{k,\ell }(x)p_{k,\ell
}(y),\, \quad x,y\in \lbrack 0,1]
\end{equation*}%
in fact exists not only for $r\in \left \{ 2,3\right \} $, but also for
arbitrary $r\in \mathbb{N}$.

In this section we use the eigenstructure of $M_n$ in order to study the
kernel function of the composition of $r$ Bernstein--Durrmeyer operators. As
a corollary we get a representation of the iterate $M_n^r$ as a linear
combination of $M_k$, $k=0,1,\dots,n$.

Let $r\in {\mathbb{N}}$, $n_1,\cdots, n_r\in {\mathbb{N}}$ and $n:=\min
\{n_1,\dots,n_r\}$.

\begin{theorem}
\label{THM4} The function 
\begin{equation*}
K_{n_r,\cdots,n_1}(x,y)=\displaystyle \sum_{k=0}^nc_k\cdot(k+1)\sum_{%
\ell=0}^kp_{k,\ell}(x)p_{k,\ell}(y),\, \, x,y\in[0,1],
\end{equation*}
where $(c_0,\dots,c_n)$ is the unique solution of the linear system 
\begin{equation}  \label{eq.l1}
\displaystyle \sum_{k=0}^n\dfrac{k^{\underline{j}}}{(k+j+1)^{\underline{j}}}%
c_k=\prod_{i=1}^r\dfrac{n_i^{\underline{j}}}{(n_i+j+1)^{\underline{j}}},\,
\, j=0,1,\dots,n,
\end{equation}
is the kernel of the operator $M_{n_r}\circ \dots \circ M_{n_1}$, i.e., 
\begin{equation*}
(M_{n_r}\circ \dots \circ
M_{n_1})f(x)=\int_0^1f(y)K_{n_r,\dots,n_1}(x,y)dy,\, \, f\in L_{1}[0,1],\,x%
\in[0,1].
\end{equation*}
\end{theorem}

\begin{remark}
The coefficients $c_{k}$, $k=0,\dots ,n$, depend on $n_{1},\dots ,n_{r}$.
\end{remark}

Note that the linear system (\ref{eq.l1}) has a unique solution since the
corresponding matrix 
\begin{equation}
A_{n}=\left( \frac{k^{\underline{j}}}{\left( k+j+1\right) ^{\underline{j}}}%
\right) _{j,k=0,\ldots ,n}  \label{eq.*}
\end{equation}%
is invertible. This can easily be seen. The square matrix $A_{n}$ is upper
triangular since $k^{\underline{j}}=0$, if $j>k\geq 0$. Hence, its
determinant satisfies 
\begin{equation*}
\det A_{n}=\prod_{k=0}^{n}\frac{k!}{\left( 2k+1\right) ^{\underline{k}}}%
=\prod_{k=0}^{n}\binom{2k+1}{k}^{-1}\neq 0.
\end{equation*}

\begin{proof}[Proof of Theorem~\protect\ref{THM4}]
Let $W$ be the integral operator on $L_{1}[0,1]$ constructed with the kernel 
$K_{n_{r},\dots ,n_{1}}$. We have 
\begin{align*}
Wf(x)& =\int_{0}^{1}K_{n_{r},\dots ,n_{1}}(x,y)f(y)dy=\displaystyle%
\int_{0}^{1}f(y)\sum_{k=0}^{n}c_{k}(k+1)\sum_{\ell =0}^{k}p_{k,\ell
}(x)p_{k,\ell }(y)dy \\
& =\displaystyle \sum_{k=0}^{n}c_{k}(k+1)\sum_{\ell =0}^{k}p_{k,\ell
}(x)\int_{0}^{1}f(y)p_{k,\ell }(y)dy \\
& =\sum_{k=0}^{n}c_{k}M_{k}f(x).
\end{align*}%
It is sufficient to prove that 
\begin{equation*}
WQ_{j}=(M_{n_{r}}\circ \dots \circ M_{n_{1}})Q_{j},\, \,j=0,1,\dots ,n.
\end{equation*}%
This is equivalent to 
\begin{equation*}
\left( \sum_{k=0}^{n}c_{k}M_{k}\right) Q_{j}=(M_{n_{r}}\circ \dots \circ
M_{n_{1}})Q_{j},\, \,j=0,1,\dots ,n,
\end{equation*}%
and moreover, to the system (\ref{eq.l1}). This concludes the proof.
\end{proof}

\begin{remark}
For $r\geq 4$, it remains an open problem whether the coefficients $c_{k}$
in Theorem \ref{THM4} can be represented in a concise form as in the cases $%
r\leq 3$.
\end{remark}

\begin{theorem}
The inverse of the matrix $A_{n}$ defined in (\ref{eq.*}) is given by 
\begin{equation*}
A_{n}^{-1}=\left( \left( -1\right) ^{k-\ell }\frac{2\ell +1}{k+1}\binom{%
k+\ell }{\ell }\binom{\ell }{k}\right) _{k,\ell =0,\ldots ,n}.
\end{equation*}
\end{theorem}

\begin{proof}
The equation $\binom{\ell }{k}=0$, for $k>\ell \geq 0$, reflects the fact
that matrix $A_{n}^{-1}$ is upper triangular.

We calculate $A_{n}A_{n}^{-1}=:\left( e_{j,\ell }\right) _{j,\ell =0,\ldots
,n}$. Noting that $k^{\underline{j}}=0$, if $0\leq k<j$, we have 
\begin{eqnarray*}
e_{j,\ell } &=&\sum_{k=j}^{\ell }\frac{k^{\underline{j}}}{\left(
k+j+1\right) ^{\underline{j}}}\cdot \left( -1\right) ^{k-\ell }\frac{2\ell +1%
}{k+1}\binom{k+\ell }{\ell }\binom{\ell }{k} \\
&=&\frac{2\ell +1}{\ell !}\sum_{k=j}^{\ell }\left( -1\right) ^{k-\ell }%
\binom{\ell }{k}\frac{k^{\underline{j}}\left( k+1\right) ^{\overline{\ell }}%
}{\left( k+1\right) ^{\overline{j+1}}} \\
&=&\frac{2\ell +1}{\ell !}\sum_{k=j}^{\ell }\left( -1\right) ^{k-\ell }%
\binom{\ell }{k}\frac{k^{\underline{j}}\left( k+j+1\right) ^{\overline{\ell
-j}}}{k+j+1}.
\end{eqnarray*}%
In the case $j=\ell $ we obtain 
\begin{equation*}
e_{\ell ,\ell }=\frac{2\ell +1}{\ell !}\frac{\ell !}{2\ell +1}=1.
\end{equation*}%
Next, we consider the case $j<\ell $: 
\begin{equation*}
e_{j,\ell }=\frac{2\ell +1}{\ell !}\sum_{k=0}^{\ell }\left( -1\right)
^{k-\ell }\binom{\ell }{k}k^{\underline{j}}\left( k+j+2\right) ^{\overline{%
\ell -j-1}}=0,
\end{equation*}%
since $k^{\underline{j}}\left( k+j+2\right) ^{\overline{\ell -j-1}}$ is a
polynomial of degree $\ell -1$ in the variable $k$. In the case $j>\ell $ we
obviously have $e_{j,\ell }=0$.
\end{proof}

Inspecting the proof of Theorem \ref{THM4} we obtain the following result
concerning the iterates of the Bernstein--Durrmeyer operators.

\begin{corollary}
Let $r,n\in {\mathbb{N}}$. Then, 
\begin{equation*}
M_n^r=\sum_{k=0}^n c_kM_k,
\end{equation*}
where $(c_0,\dots,c_n)$ is the unique solution of the system 
\begin{equation*}
\displaystyle \sum_{k=0}^n\dfrac{k^{\underline{j}}}{(k+j+1)^{\underline{j}}}%
c_k=\left(\dfrac{n^{\underline{j}}}{(n+j+1)^{\underline{j}}}\right)^r,\, \,
j=0,1,\dots,n.
\end{equation*}
\end{corollary}

\section{The case of Sz\'{a}sz--Mirakjan--Durrmeyer operators}

\label{sec4}

For $\alpha \geq 0$ we denote by $W_{\alpha }$ the space of all locally
integrable functions on $[0,\infty )$, satisfying for $t\geq 0$ an
exponential growth condition, i.e., $|f(t)|\leq Ce^{\alpha t}$ for some
positive constant $C$. For $f\in W_{\alpha }$ the Sz\'{a}%
sz--Mirakjan--Durrmeyer operators are defined by 
\begin{align*}
& S_{n}f(x)=\int_{0}^{\infty }f(y)K_{n}(x,y)dy, \\
& K_{n}(x,y)=n\sum_{k=0}^{\infty }s_{n,k}(x)s_{n,k}(y),\, \,s_{n,k}(x)=%
\dfrac{(nx)^{k}}{k!}e^{-nx}.
\end{align*}

The commutativity of the Sz\'{a}sz--Mirakjan--Durrmeyer operators was
established by Heilmann \cite{Hei}. The following general result,
encompassing the commutativity and extending beyond it, was later proved by
Abel and Ivan in \cite{Abel-Ivan-Varna-CTF2005-proc-2006} using modified
Bessel functions. The same result was rediscovered in \cite[Theorem 1]{S}.

%The commutativity of the Szász–Mirakjan–Durrmeyer operators was first established by Heilmann \cite{Hei}. A more general result, encompassing this commutativity and extending beyond it, was later proved by Abel and Ivan in \cite{Abel-Ivan-Varna-CTF2005-proc-2006} using Bessel functions. This result was subsequently rediscovered in \cite[Theorem 1]{S}.

%The following result was originally proved by Abel and Ivan in \cite%
%{Abel-Ivan-Varna-CTF2005-proc-2006} using Bessel functions, with the
%commutativity of the operators arising as a corollary. The same result was
%later rediscovered in \cite[Theorem %1]{S}. 

Here we present an elementary proof.

\begin{theorem}
\label{Thm8} For $m,n> 0$, the kernel of the composition of Sz\'{a}%
sz--Mirakjan--Durrmeyer operators $S_{m}\circ S_{n}$ operators has the
representation 
\begin{equation*}
K_{m,n}\left( x,y\right) =K_{\frac{mn}{m+n}}\left( x,y\right).
\end{equation*}
\end{theorem}

The corresponding well-known formula for the basis functions are 
\begin{align}
& \sum_{k=0}^{\infty }s_{n,k}(x)=1, \nonumber  \\  %%%\label{eq.sm1} \\
& \int_{0}^{\infty }s_{n,k}\left( t\right) dt=\frac{1}{n}, \text{ \quad }
k=0,1,2,\ldots ,  \label{eq.sm2} \\
& s_{m,\mu }\left( t\right) s_{n,\nu }\left( t\right) =\frac{m^{\mu }n^{\nu
} }{(m+n)^{\mu +\nu }}\binom{\mu +\nu }{\nu }s_{m+n,\mu +\nu }\left(
t\right) .  \label{eq.sm3}
\end{align}

\begin{proof}[Proof of Theorem~\protect\ref{Thm8}]
Using (\ref{eq.sm2}) and (\ref{eq.sm3}) we get 
\begin{equation*}
K_{m,n}(x,y)=\displaystyle \frac{mn}{m+n}\sum_{k\geq 0}s_{m,k}(x)\left( 
\dfrac{m}{m+n}\right) ^{k}\dfrac{1}{k!}\sum_{\ell \geq 0}s_{n,\ell
}(y)\left( \dfrac{n}{m+n}\right) ^{\ell }(k+\ell )^{\underline{k}}.
\end{equation*}%
Writing 
\begin{equation*}
(k+\ell )^{\underline{k}}=\left. \left( \dfrac{d}{ds}\right) ^{k}s^{k+\ell
}\right \vert _{s=1},
\end{equation*}%
we obtain for the inner sum 
\begin{eqnarray*}
&&\displaystyle \sum_{\ell \geq 0}s_{n,\ell }(y)\left. \left( \dfrac{n}{m+n}%
\right) ^{\ell }\left( \dfrac{d}{ds}\right) ^{k}s^{k+\ell }\right \vert
_{s=1} \\
&=&e^{-ny}\left( \dfrac{d}{ds}\right) ^{k}\left. \left \{ s^{k}\sum_{\ell
\geq 0}\dfrac{1}{\ell !}{\left( \dfrac{n^{2}ys}{m+n}\right) ^{\ell }}\right
\} \right \vert _{s=1} \\
&=&e^{-ny}\left( \dfrac{d}{ds}\right) ^{k}\left. \left \{ s^{k}\exp \left( 
\dfrac{n^{2}ys}{m+n}\right) \right \} \right \vert _{s=1} \\
&=&e^{-ny}\sum_{j=0}^{k}{\binom{k}{j}}\dfrac{k!}{j!}\left( \dfrac{n^{2}y}{m+n%
}\right) ^{j}\exp \left( \dfrac{n^{2}y}{m+n}\right) ,
\end{eqnarray*}%
where the last equation follows by Leibniz' rule. Thus we get 
\begin{eqnarray*}
K_{m,n}(x,y) &=&\dfrac{mn}{m+n}\exp \left( -ny-mx+\dfrac{n^{2}y}{m+n}\right)
\\
&&\sum_{k\geq 0}\left( \dfrac{m^{2}x}{m+n}\right) ^{k}\dfrac{1}{k!}%
\sum_{j=0}^{k}{\binom{k}{j}}\dfrac{1}{j!}\left( \dfrac{n^{2}y}{m+n}\right)
^{j} \\
&=&\dfrac{mn}{m+n}\exp \left( -ny-mx+\dfrac{n^{2}y}{m+n}\right) \sum_{j\geq
0}\dfrac{1}{j!j!}\left( \dfrac{n^{2}y}{m+n}\right) ^{j} \\
&&\sum_{k\geq j}\left( \dfrac{m^{2}x}{m+n}\right) ^{k}\dfrac{1}{(k-j)!} \\
&=&\dfrac{mn}{m+n}\exp \left( -ny-mx+\dfrac{n^{2}y}{m+n}+\dfrac{m^{2}x}{m+n}%
\right) \\
&&\sum_{j\geq 0}\dfrac{1}{j!j!}\left( \dfrac{n^{2}y}{m+n}\right) ^{j}\left( 
\dfrac{m^{2}x}{m+n}\right) ^{j} \\
&=&\dfrac{mn}{m+n}\exp \left( -\dfrac{mn}{m+n}x-\dfrac{mn}{m+n}y\right)
\sum_{j\geq 0}\dfrac{1}{j!}\left( \dfrac{mn}{m+n}x\right) ^{j}\dfrac{1}{j!}%
\left( \dfrac{mn}{m+n}y\right) ^{j} \\
&=&\dfrac{mn}{m+n}\sum_{j\geq 0}s_{\frac{mn}{m+n},j}(x)s_{\frac{mn}{m+n}%
,j}(y),
\end{eqnarray*}
which completes the proof.
\end{proof}

\begin{remark}
This can be iterated 
\begin{equation*}
K_{n_{1},\ldots ,n_{r}}\left( x,y\right) =K_{1/\left( 1/n_{1}+\cdots
+1/n_{r}\right) }\left( x,y\right)
\end{equation*}%
and be rewritten in the form 
\begin{equation*}
K_{n_{1},\ldots ,n_{r}}\left( x,y\right) =K_{H\left( n_{1},\ldots
,n_{r}\right) /r}\left( x,y\right) ,
\end{equation*}%
where $H\left( n_{1},\ldots ,n_{r}\right) $ denotes the harmonic mean of the
positive numbers $n_{1},\ldots ,n_{r}$.
\end{remark}

\begin{remark}
In particular, the iterate $S_{n}^{r}$ has the kernel 
\begin{equation*}
K_{n,\ldots ,n}\left( x,y\right) =K_{n/r}\left( x,y\right) .
\end{equation*}
\end{remark}
Of course, $K_{m,n}\left( x,y\right) $ can be represented in terms of Bessel
functions (see \cite{Abel-Ivan-Varna-CTF2005-proc-2006}) 
\begin{equation*}
K_n(x,y)=ne^{-n(x+y)}\displaystyle \sum_{k=0}^{\infty}\dfrac{(nx)^k}{k!}%
\dfrac{(ny)^k}{k!}=ne^{-n(x+y)}I_0(2n\sqrt{xy}),
\end{equation*}
where $I_0$ denotes the modified Bessel function given by 
\begin{equation*}
I_0(z)=\displaystyle \sum_{j=0}^{\infty}(j!)^{-2}\left(\dfrac{z^2}{4}%
\right)^j.
\end{equation*}

Our derivation does not make use of properties of Bessel functions, in
particular, integral representations.

%%%%%%%%%%%%%%%%%%%%%%%%%%%%%%%%%%%%%%%%%%%%%%%%%%%%%%%%%%%%%%%%%%%

\strut

\noindent
{\bf Funding.} The work of A.M. Acu and I. Rasa was supported by the project
``Mathematical Methods and Models for Biomedical Applications'' financed by
National Recovery and Resilience Plan PNRR-III-C9-2022-I8.

\strut

\thispagestyle{empty}

~\vfill

\end{document}